\documentclass[11pt,leqno]{amsart}

\usepackage{amscd}
\usepackage{color}
\usepackage{amssymb}
\usepackage{amsfonts}
\usepackage{latexsym}
\usepackage{verbatim}
\usepackage{pb-diagram}

\theoremstyle{plain}
\newtheorem{theorem}{Theorem}[section]
\newtheorem{definition}[theorem]{Definition}
\newtheorem{lemma}[theorem]{Lemma}
\newtheorem{prop}[theorem]{Proposition}
\newtheorem{cor}[theorem]{Corollary}

\newtheorem{ex}[theorem]{Example}

\renewcommand{\b}{\begin{equation}}
\newcommand{\e}{\end{equation}}

\newcommand\C{{\mathbb C}}
\newcommand\R{{\mathbb R}}

\newcommand{\N}{\nabla}

\newcommand{\F}{\mathcal{F}}
\begin{document}
\title[Second order geometric flows on foliated manifolds]{Second order geometric flows on foliated manifolds}
\date{\today}
\author[L.~Bedulli]{Lucio Bedulli}
\address{Dipartimento di Ingegneria e Scienze dell'Informazione e Matematica \\ Universit\`a dell'Aquila\\
via Vetoio\\ 67100 L'Aquila\\ Italy}
\email[L.~Bedulli]{bedulli@math.unifi.it}

\author[W.~He]{Weiyong He}
\address{Department of Mathematics, University of Oregon, Eugene, Oregon, 97403}
\email[W.~He]{whe@uoregon.edu}

\author[L. ~Vezzoni]{Luigi Vezzoni}
\address{Dipartimento di Matematica \\ Universit\`a di Torino\\
Via Carlo Alberto 10\\
10123 Torino\\ Italy} 
\email[L.~Vezzoni]{luigi.vezzoni@unito.it}

\thanks{This work was supported by the project FIRB ``Geometria differenziale e teoria geometrica delle funzioni'',
 the project PRIN
\lq\lq  Variet\`a reali e complesse: geometria, topologia e analisi armonica" and by G.N.S.A.G.A. of I.N.d.A.M}
\begin{abstract}
We prove a general result about the short time existence and uniqueness of second order geometric flows transverse to a Riemannian foliation on a compact manifold. Our result includes some flows already existing in literature, as the {\em transverse Ricci flow}, the {\em Sasaki-Ricci flow} and the {\em Sasaki $J$-flow} and motivates the study of other evolution equations. We also introduce 
a transverse version of the {\em K\"ahler-Ricci flow} adapting some classical results to the foliated case. 
\end{abstract}

\maketitle

\section{Introduction}
In this paper we study transversally parabolic flows on manifolds foliated by Riemannian foliations. The definition of  {\em Riemannian foliation} was introduced 
and firstly studied by B. Reinhart in \cite{Reinhart} as  a natural generalization of
Riemannian submersions. Roughly speaking a foliation $\F$ on a manifold $M$ is Riemannian if there  exists a Riemannian metric on $M$ such that the distance from one leaf of $\F$ to another is 
 locally constant. The normal bundle $Q$ to a Riemannian foliation $\F$ inherits a metric  $g_Q$ along the fibres which is \lq\lq constant'' along the leavesÓ of $\mathcal F$. Furthermore, $g_Q$ induces a canonical connection $\nabla$ on $Q$ preserving $g_Q$ and having vanishing transverse torsion. This connection can be used to define the transverse curvature and the transverse Ricci tensor of $g_Q$. 

Searching for a preferred transverse metric on a manifold foliated by a Riemannian foliation, it is quite natural to follow the nonfoliated case studying the flow of a transverse metric along the transverse Ricci tensor.
This was initiated in \cite{Min-Oo} in the context of Cartan geometry where it is introduced the {\em transverse Ricci flow} and it is proved a foliated version of the famous Hamilton's results for 3-dimensional compact manifolds with positive Ricci tensor (see \cite{positive}). Furthermore, the {\em transverse Ricci flow} was used in \cite{SWZ}  to evolve  Sasakian metrics and then investigated in \cite{collins0,collins1,collins2,collins3,he,wang}. A similar flow for evolving Riemannian metrics on manifolds foliated by $1$-codimensional non-Riemannian foliations was introduced and studied in \cite{rovenski,rovenski2}. 

For the flows mentioned above, the short-time existence is proved by using an argument ad hoc. For instance, in \cite{Min-Oo} the short-time existence of the transverse Ricci flow is obtained regarding the flow as a flow of Cartan connections and then applying the original technique of Hamilton for parabolic systems satisfying integrability conditions, whilst in \cite{SWZ} the short  time existence of the Sasaki Ricci flow is obtained  by modifying the flow with a \lq\lq parabolic complement\rq \rq.  

The main goal of this paper is to show that a second order quasilinear transversally parabolic flow of basic sections of a vector bundle over a foliated manifold has always a unique short-time solution. This result implies the short time existence of the transverse Ricci flow and of the Sasaki-Ricci flow and motivates the study of other flows. 
The precise framework is the following: we consider a compact manifold $M$ foliated by a {\em transversally orientable} Riemannian foliation $\F$, an $\mathcal F$-bundle $(E,\nabla)$ over $M$ and a second order quasilinear basic partial differential operator 
$$
D\colon C^{\infty}(E/\mathcal F)\to  C^{\infty}(E/\mathcal F)\,.
$$
By an \lq\lq{\em $\mathcal F$-bundle}\rq\rq\,  we mean a vector bundle $\pi\colon E\to M$ with an assigned connection $\nabla$ whose curvature vanishes along vector fields tangent to the the foliations. Furthermore, $C^{\infty}(E/\mathcal F)$ denotes the set of smooth sections $u$ of $E$  satisfying $\N_Xu=0$ for every vector field $X$ tangent to $\mathcal F$.  Roughly speaking, $D$ is a {\em basic partial differential operator} if locally with respect to a foliated atlas it  writes as a partial differential operators in the transverse coordinates  (see definition \ref{defD}).  
In this set-up, we consider the evolution equation
\begin{equation}\label{parabolic}
\partial_t u_t=D(u_t)\,,\quad 
u_{|t=0}=u_0
\end{equation}
where $u_0\in C^{\infty}(E/\F)$ is fixed and the solution $u\colon M\times [0,\epsilon)\to M$ is required to be smooth and such that $u_t\in C^{\infty}(E/\mathcal F)$ for every $t$.  
\begin{theorem}\label{existence}
Assume that $D$ is strongly transversally elliptic at $u_0$. Then equation \eqref{parabolic} has always a unique maximal solution defined for $t\in [0,\epsilon)$. Moreover, when $D$ is linear $u$ is defined for  
$t\in [0,\infty)$. 
\end{theorem}
The proof of theorem \ref{existence} is mainly based on the treatment in \cite{EKA} of basic differential operators. 
Indeed, from \cite{molino,EKA} it follows that if $(E,\nabla)$ is an $\F$-bundle over a manifold $M$ foliated by a transversally oriented Riemannian foliation $\F$, then there exist a compact smooth manifold $W$ and an ${\rm SO}(n)$-bundle $\bar E$ over $W$ such that 
$C^{\infty}(E/\F)$ is canonically isomorphic to the space of ${\rm SO}(n)$-invariant sections of $\bar E$.   
Moreover, from \cite{EKA} it follows that if  $D\colon C^{\infty}(E/\mathcal F)\to  C^{\infty}(E/\mathcal F)$ is a {\em linear} {\em basic} strongly transversally elliptic operator, then it can be regarded as a $G$-invariant strongly elliptic differential  operator on $C^{\infty}(\bar E)$. From this result it follows that in the linear case equation  \eqref{parabolic} can be regarded as a genuine parabolic equation involving sections of a fiber bundle and the existence and uniqueness of a solution follows from the standard parabolic theory. The proof of the nonlinear case follows the same approach, but since the results in \cite{EKA} are proved only for linear operators, we have to adapt El Kacimi's theorem to the nonlinear case (see theorem \ref{ourEKA}). 

In the second part of the paper we apply theorem \ref{existence} to some explicit flows on foliated manifolds. In section $5$ we consider the {\em transverse Ricci flow} introduced in \cite{Min-Oo} and we prove that it is  well-posed by applying theorem \ref{existence}. Indeed, as it happens in the non foliated case, the flow is not strongly parabolic and it has to be modified by using a basic vector field. The modified transverse Ricci flow is strongly transversally parabolic and it is well-posed in view of 
theorem \ref{existence}. The existence of a solution to the transverse Ricci flow follows from the well-posedness of its modification, while for the uniqueness of the solutions we adapt an argument in \cite{Kotschwar} to the foliated case. In this second part we have to assume that the foliation is {\em homologically orientable} in order to introduce an integral functional.  

In section \ref{kahler1} we take into account K\"ahler foliations studying a foliated version of the K\"ahler-Ricci flow.  Here the well-posedness of the flow is again implied by theorem \ref{existence}, while the long time existence is obtained adapting some well-known results of the non-foliated case. 


\medskip
\noindent {\em{Acknowledgements}}. The authors would like to thank professor Min-Oo for useful comments and remarks. 
\section{Preliminaries on Riemannian foliations}
Let $M$ be an $(m+n)-$dimensional smooth manifold. A codimension $n$ foliation $\F$ on $M$ can be defined
as an open cover $\{U_i\}$ of $M$ together a family $f_i\colon U_i\to T$ of submersions onto an $m$-dimensional manifold $T$, called the  {\em base} of the foliation, such that whenever $U_i\cap U_j \neq \emptyset$ there exists a smooth map $\gamma_{ij}\colon f_i( U_i\cap U_j)\to f_j(U_i\cap U_j)$ satisfying $f_j \circ \gamma_{ij}=f_i$. The pair $(M,\F)$ is usually called a {\em foliated manifold}. 
The previous construction is equivalent to assign an {\em involutive} distribution $L$ of rank $m$. A {\em leaf} of $\mathcal F$ is by definition a maximal integral submanifold of $L$. Given a foliated manifold $(M,\F)$, we denote by $Q$ the normal bundle $TM/L$ and it is therefore defined the following exact sequence of vector bundles 
$$
0 \rightarrow L \rightarrow TM\rightarrow Q\to 0\,.
$$     

 A {\em transverse structure} on a foliated manifold is by definition a geometric structure on the base manifold $T$ which is invariant by the transition maps $\{\gamma_{ij}\}$. The most important class of transverse structures is the one of {\em transverse Riemannian metrics} introduced by Reinhart in \cite{Reinhart} as a natural generalisation of Riemannian submersions.  In contrast to the non-foliated case, the existence of a transverse Riemannian metric is not always guaranteed and imposes some strong conditions on the foliation. A foliation $\F$ is Riemannian if and only if $Q$ inherits a metric $g_Q$ along its fibres satisfying the {\em holonomy invariant} condition 
\b\label{holinv}
\mathcal{L}_Xg_{Q}=0
\e 
for every $X\in C^{\infty}(L)$, where $\mathcal{L}$ is the Lie derivative (see e.g. \cite{T,molino}). Condition
\eqref{holinv} makes sense since if $X$ is a section of $L$, then its flow $\{\phi_t\}$ preserves the foliation. 
The metric $g_Q$ is simply defined by gluing together the pull-backs of the metric of the base $T$ via the local submersions and can be regarded as a degenerate symmetric $2$-tensor on $M$. Such a $g_Q$ can be always \lq\lq completed" to a bundle-like Riemannian metric $g$ on $M$, i.e. there always exists a metric $g$ on $M$ whose restriction to the orthogonal complement of $L$ is $g_Q$. A metric $g$ on a foliated manifold is usually called {\em bundle-like} if its restriction to $L$ satisfies \eqref{holinv}. 
In this paper we refer to a metric $g_Q$ on $Q$ satisfying \eqref{holinv} as to a {\em transverse metric}. 

A vector field $X$ on a foliated manifold $(M,\F)$ is called {\em foliated} if $[X,Y]\in C^{\infty}(L)$ for every $Y\in C^{\infty}(L)$. Denoting by $C^{\infty} (M,\F)$ the space of foliated vector fields on $(M,\F)$, the quotient  $C^{\infty}(M/\F):=C^{\infty}(M,\F)/C^{\infty}(L)$ is by definition the set of {\em basic vector fields} on $(M,\F)$. For every $X\in C^{\infty}(M,\F)$, $X^b$ denotes the correspondent class in $C^{\infty}(M/\F)$ and can be regarded as a section of $Q$. A foliation is called {\em transversally parallelizable} if there exist $n$ basic vector fields which are linear independent at any point.   


Now we recall the definition of {\em $\F$-fibration.} Let $\pi\colon P\to M$ be a $G$-principal bundle over a foliated manifold $(M,\F)$ and let $H \subseteq TP$ be the horizontal distribution defined by a connection on $P$ with connection $1$-form $\omega$. Then for every $p\in P$, $H_p$ is isomorphic to $T_{\pi(p)}M$ and consequently $\F$ induces a distribution $\tilde \F$ in $P$. The connection given by $H$ is called a {\em basic} if $\tilde \F$ is a foliation and $\omega$ is basic, i.e. $\mathcal L_{\tilde X} \omega=0$ for every vector field $\tilde X$ tangent to $\tilde{\F}$. In this case the pair $(P,H)$ is called an {\em $\F$-fibration}.  A vector bundle $E$ with an assigned affine connection $\nabla$ is called an $\F$-{\em fibration} if the induced principal bundle $(P,H)$ is an $\F$-fibration. This is equivalent 
to require that the curvature $R$ of $\nabla$ satisfies $\iota_XR=0$ for every smooth section $X$ of $L$. Moreover, if $(E,\nabla)$ is an $\F$-bundle, then the foliation $\tilde \F$ on the principal bundle induces a foliation $\F_E$ on $E$.  A map $T$ between two $\F$-bundle $(E,\nabla)$ and $(E',\nabla')$ on $M$ is called {\em foliated} if it takes leaves of $\F_E$ to leaves of $\F_{E'}$ and a smooth section $\alpha$ of an $\F$-bundle $(E,\nabla)$ is called {\em basic} if $\N_Xu=0$ for every $X$ tangent to $\F$. We denote by $C^{\infty}(E/\F)$ the set of smooth basic sections of $(E,\nabla)$.

The most natural example of $\F$-fibration is the ${\rm SO}(n)$-bundle  of transversally oriented frames of a Riemannian foliation defined as follows. 
Let $\F$ be a Riemannian foliation on a smooth manifold $M$ with transverse metric $g_Q$. Then it is defined the {\em transverse Levi-Civita} connection $\N$ on $Q$ as 
\begin{equation}\label{TLC}
\nabla_{X}V=
\begin{cases}
\begin{array}{ccl}
[X,\sigma(V)]_{Q}& \mbox{if}& X\in \Gamma(L)\\
(\N^g_{X}\sigma(V))_{Q}& \mbox{if} &X\in \sigma(Q),
\end{array}
\end{cases}
\end{equation}
for every $V\in C^{\infty}(Q)$, where $g$ is a bundle-like metric on $M$ inducing $g_Q$ with Levi-Civita connection $\nabla^g$, $\sigma$ is the isomorphism between $Q$ and $L^{\perp}$ and $X\mapsto X_Q$ is the projection onto $Q$. Such a $\nabla$ does not depend on the choice of $g$ and it is the unique connection on $Q$ satisfying 
\begin{eqnarray}
&&\label{preservingg} Xg_Q(V_1,V_2)=g_Q(\nabla_X V_1,V_2)+g_Q(V_1,\nabla_XV_2)\\
&&\label{free} \nabla_X Y_Q - \nabla_Y X_Q -[X,Y]_Q = 0\,, 
\end{eqnarray}  
for every vector fields $X,Y$ on $M$ and $V_1$ and $V_2$ in  $C^{\infty}(Q)$. Moreover, $\F$ is called {\em transversally orientable} if there exists a nowhere vanishing transverse volume form $\nu$. When such a $\nu$ is fixed, then the principal bundle of linear frames of $Q$ has a natural ${\rm SO}(n)$-reduction which we denote as in \cite{molino} by $\rho\colon M^{\sharp}\to M$. The transverse Levi-Civita connection induces a connection on $M^{\sharp}$ making it an $\F$-bundle. The following result is due to Molino and it will be important subsequently:

\begin{theorem}[Molino, \cite{molino}]\label{molinoth}
The foliation $\F^{\sharp}$ induced by $\F$ on $M^{\sharp}$ is always transversally parallelizable and invariant by the action of ${\rm SO}(n)$. Moreover, the leaf closures of $\F^{\sharp}$ are the fibres of a locally trivial fibration $F\hookrightarrow M^{\sharp}\to W$, where $W$ is a compact manifold called the {\em basic manifold} of $\F$. 
\end{theorem}

In the last part of this section we recall the definition of the {\em basic Laplace operator} and {\em basic cohomology} groups. 
Let $(M,\F)$ be a manifold foliated by a Riemannian foliation and let $g_Q$ be its transverse metric. A $p$-form $\alpha$ on $M$ is called {\em basic} if 
$$
\iota_{X}\alpha=0\,,\quad \mathcal{L}_X\alpha=0
$$ 
for every smooth section $X$ of the fibre bundle $L$ generated by $\F$, where $\mathcal{L}$ denotes the Lie derivative.  We denote by $\Omega^p_B(M)$ the set of basic $p$-form on $M$ and by $C^{\infty}_B(M)$ the set of basic fuctions (i.e. basic $0$-forms). Notice that accordingly to our previous notation we have $C^{\infty}_B(M)=C^{\infty}(M\times \R/\mathcal F)$. 
 Then the de Rham differential operator takes  basic forms into basic forms and the pair $(\Omega_B,d)$ induces a cohomology $H_B$ usually called the {\em basic cohomology} of $(M,\F)$.  As is usual we will denote by $d_B$ the restriction $d_{|\Omega_B}$. When $\F$ is transversally oriented the basic hodge \lq\lq star\rq\rq operator $*_B$ and the {\em basic codifferential operator}  $\delta_B$ are defined in the usual way. Furthermore, it is defined the {\em basic Laplacian operator}  $\Delta_B=d_B\delta_B+\delta_B d_B$ acting on basic forms of degree at least 1. On the other hand for conventional reasons we put a minus sign in the definition of the basic Laplacian acting on basic functions $\Delta_B=-d_B\delta_B+\delta_B d_B: C^{\infty}_B(M) \to C^{\infty}_B(M)\,.$ 
As in the classical Hodge theorem, in the compact case the basic cohomology groups are isomorphic to the kernels of $\Delta_B$, but, in contrast to  the nonfoliated case, they do not always satisfy Poincar\'e duality.  Poincar\'e duality is guaranteed under some strong topological assumption on $\F$, for instance when $\F$ is homological orientable: 
 
\begin{definition}\label{homorientable}
A  transversally oriented Riemannian foliation $\F$ is called {\em homologically orientable} if
there exists an $m$-form $\chi$ on $M$ restricting to a volume along the leaves of $\F$ and such that 
$$
\iota_X d\chi=0
$$
for every $X\in C^{\infty}(L)$.  
\end{definition}
It is well-known that when $\F$ is homologically orientable, the form $\chi$ can alwaysbe  written as  
$$
\chi(Y_1,\dots,Y_m)=\det\left(g(Y_i,E_j)\right)\,,\quad Y_1,\dots ,Y_m\in \Gamma(TM)
$$
where $g$ is a bundle-like metric on $M$ making the leaves of $\F$ minimal and 
 $\{E_1,\dots,E_m\}$ is an oriented orthonormal frame of $L$. Furthermore, the existence of $\chi$ allows us to introduce the following scalar product on basic forms  
\b\label{scalar}
(\alpha,\beta)=\int_M \alpha\wedge *_B\beta\wedge \chi\,.
\e
which makes $\Delta_B$ self-adjoint.  
%
%

\section{Basic differential operators on foliated manifolds}\label{basicdiffoperfolaitedmanifolds}
In order to introduce {\em basic differential operators} on foliated manifolds, we briefly recall the non-foliated case. Let $M$ be a manifold and let $\pi\colon E\to M$ be a vector bundle over $M$. We denote by $C^{\infty}(E)$ the vector space of smooth global sections of $E$. A {\em quasilinear differential operator} of order $r$ is a map $D\colon C^{\infty}(E)\to C^{\infty}(E)$ which can be locally written as 
$$
D(u)=\left[a^{i_1\dots i_r}_{\alpha\beta}(x,u,\nabla u,\dots,\nabla^{r-1} u)\partial_{x^{i_1}\cdots x^{i_r}}u^{\beta}
+b_{\alpha}(x,u,\nabla u,\dots,\nabla^{r-1} u)\right]\,e_{\alpha}
$$
where  $\{x^{r}\}$ are local coordinates on $M$ and $\{e_{\alpha}\}$ is a local  frame of $E$. In this definition and throughout all the paper we use the Einstein summation convention. When $D$ has even order $r$, it is called {\em strongly elliptic} at $u\in C^{\infty}(E)$ if there exists a constant $\lambda>0$ such that the differential $L_{u}=D_{|*u}$ of
$D$ at $u$ satisfies
\b\label{se}
(-1)^{r/2}h(\sigma(L_{u})(x,\xi)v,v)\geq \lambda |\xi|^2 |v|^2
\e
for every $(x,\xi)\in T^*M$, $\xi\neq 0$, and $v\in E_x$, where $h$ is an arbitrary metric along the fibres of $E$. 
Here
$\sigma(L_{u})$ denotes the principal symbol of $L_{u}$  which,
for every $(x,\xi)\in T^*M$, is the endomorphism of $E_x$ defined by 
$$
\sigma(L_{u_0})(x,\xi)v=\frac{i^r}{r!}L_{u_0}(f^{r}u)(x)
$$
for an $ f\in C^{\infty}(M)$ such that $f(x)=0$, $f_{|*x}=\xi$, $u\in C^{\infty}(E)$, $u(x)=v$. 
More generally, if $\tau$ is a subbundle of  $T^*M$, $D$ is called {\em strongly $\tau$-elliptic} if  $\sigma(L_{u_0})(x,\xi)$ satisfies \eqref{se} for every $(x,\xi)\in \tau$.  We recall the following classical result 
(see e.g. \cite[Chapter 4]{aubin})
\begin{theorem}\label{classical existence}
Let $D\colon C^{\infty}(E)\to C^{\infty}(E)$ be a second order quasilinear operator which is strongly elliptic at $u_0$, then the evolution equation 
$$
\partial_t u_t=D(u_t)\,\quad u_{|t=0}=u_0
$$
has a unique maximal solution 
$u\in C^{\infty}(M\times [0,\epsilon))$ for some $\epsilon>0$. Moreover, when $D$ is linear, $u$ is in $C^{\infty}(M\times [0,\infty))$.
\end{theorem}

Consider now a Lie group $G$ together a representation of $G$ in ${\rm Aut}(E).$ In this case we have also an induced $G$-action on $M$ and $E$ is usually called a {\em $G$-bundle} and $C^{\infty}(E)$ inherits the natural $G$-action $(g\cdot \alpha)(x):=g\cdot \alpha(g^{-1}x)$. We denote, adopting the notation of \cite{EKA}, by $C^{\infty}_G(E)$ the space of $G$-invariant sections of $E$. A section  $u$ of $E$ belongs to $C^{\infty}_G(E)$ if and only if $\mathcal L_Xu=0$ for every fundamental vector field $X$ of the action of $G$, where $\mathcal L$ denotes the Lie derivative. Moreover, a partial differential operator $D\colon C^{\infty}(E)\to C^{\infty}(E)$ is called {\em $G$-invariant} if it commutes with $\mathcal L_X$ for every fundamental vector field $X$. The following lemma will be useful
\begin{lemma}\label{lemmaginv}
Let $E\to  M$ be a $G$-bundle over a compact manifold, $D\colon C^{\infty}(E) \to C^{\infty}(E)$ a quasilinear second order strongly elliptic differential operator and 
 $\bar u_0$ be a $G$-invariant section of $\bar E$. Then the solution to the parabolic system 
\begin{equation}\label{parabolic2}
\partial_t u_t=D(u_t)\,,\quad u_{|t=0}=\bar u_0
\end{equation}
stays $G$-invariant for every $t$.
\end{lemma}
\begin{proof}
Let $X$ be a fundamental vector field for the action of $G$ on $M$. Then taking the Lie derivative of \eqref{parabolic2} and taking into account that $D$ commutes with $\mathcal{L}_X$ we get 
$$
\partial_t(\mathcal{L}_X  u_t)= D(\mathcal{L}_X  u_t)\,,\quad  
(\mathcal{L}_X u_{|t=0})=0\,. 
$$
Hence $\mathcal{L}_X u_t$ is a solution to 
$$
\partial_tv_t=D(v_t)\,,\quad  v_{|t=0}=0\,.
$$
and the claim follows. 
\end{proof}

Now we can focus on the foliated case. We adopt the following  definition: 
\begin{definition}\label{defD}
Let $(M,\F)$ be a foliated manifold and $(E,\nabla)$ an $\F$-fibration. A {\em quasilinear basic differential operator of order $r$} is a map   $D\colon C^{\infty}(E/\F)\to C^{\infty}(E/\F)$ such that with respect to  
local foliated coordinates $\{x^1, \dots , x^n, y^1, \dots, y^m\}$ takes the local expression
$$
D(u)=\left[a^{i_1\dots i_r}_{\alpha\beta}(y,u,\nabla u,\dots,\nabla^{r-1} u)\partial_{y^{i_1}\cdots y^{i_r}}u^{\beta}
+b_{\alpha}(y,u,\nabla u,\dots,\nabla^{r-1} u)\right]\,e_{\alpha}
$$
where $\{e_\alpha\}$ is a local trivialisation of $E$.  
\end{definition} 

When $D$ is linear, definition \ref{defD} agrees to the one given in \cite{EKA}. For a linear 
basic differential operator $D$ of order $r$ and $(x,\xi)\in T_x^*M$,  the principal symbols $\sigma(D)(x,\xi)$  of $D$ at $(x,\xi)$ is defined by*  
$$
\sigma(D)(x,\xi)v=\frac{i^r}{r!}D(f^{r}u)(x)
$$
for $v\in E_x$ and $f\in C^{\infty}(M)$ basic and such that $f(x)=0$, $f_{*|x}=\xi$, $u\in C^{\infty}(E/F)$, $u(x)=v$. In analogy to the nonfoliated case, $D$ is called {\em strongly transversally elliptic} at $u\in C^{\infty}(E/\F)$ if $D$ has even order $r$ and 
there exists a constant $\lambda>0$ such that the differential $L_{u}=D_{|*u}$ of
$D$ at $u$ satisfies
\b
(-1)^{r/2} h(\sigma(L_{u})(x,\xi)v,v)\geq \lambda |\xi|^2 |v|^2
\e
for every $(x,\xi)\in T^*M$, $\xi\neq 0$, and $v\in E_x$, where $h$ is some metric along the fibres of $E$. 

\begin{ex}
{\em The foremost example of  strongly transversally elliptic operator is the basic Laplacian operator $\Delta_B$ acting on basic functions described in the previous section.} 
\end{ex}

In analogy to the non-foliated case, every {\em linear} basic differential operator can be described in terms of jets. We briefly recall this description and refer to \cite{EKA} for details.
Let $r$ be a positive integer and let $J^r(E/\F)$ be the vector bundle whose fiber at a point $x\in M$  is given by 
$$
J^r_x(E/\F)=\frac{C^{\infty}(E/\F)}{Z_x(E/\F)}
$$
$Z_x(E/\F)$ being the ring of basic sections $u$ of $E$ satisfying $\nabla^ku=0$ at $x$ for every $k\leq r$. Let 
$$
S^k(Q, E):=S^k(Q^*)\otimes E\,,
$$
where $S^k$ denotes the $k$-symmetric power.  Then we have the canonical isomorphism 
$$
J^r(E/\F)\simeq  \oplus_{k=0}^r S^k(Q, E)
$$
(see \cite{EKA}, corollary 2.3.7). 
In particular, $J^{r}(E/\F)$ inherits a basic connection $\nabla^J$ since all the bundles involved are indeed $\mathcal F$-bundles.
  For every basic section $u$ of $E$ we denote by $J_r(u)_x$  the corresponding class in $J_r(E/\F)$.  Then it is defined the natural map $J_r \colon C^{\infty}(E/\F)\to C^{\infty}(J^{r}(E/\F))$ as $J_r(f)(x):=J_r(f)_x$. Every linear basic partial differential operator $D\colon C^{\infty}(E/\F)\to C^{\infty}(E/\F)$ of order $r$ can be written as $D=T'\circ J_r$, where 
$T'\colon C^{\infty}(J^r(E))\to C^{\infty}(E)$ is the map induced by the
{\em foliated} morphism $T\colon J^{r}(E/\F)\to E$.

%
%

\section{Proof of theorem \ref{parabolic}}
This section contains the proof of theorem  \ref{existence} and it is subdivided in two parts. The first part is about the linear case and it is obtained as direct consequence of a theorem of El Kacimi proved in \cite{EKA} (see theorem \ref{ekatheorem} below). For the nonlinear case, we generalise El Kacimi's theorem to quasilinear operators and then we get the proof of theorem \ref{parabolic} as a consequence. 

\medskip 
Let us consider a compact manifold $M$ equipped with an $n$-codimensional transversally oriented Riemannian foliation $\F$ and let $(E,\nabla)$ be an $\F$-bundle over $M$. Let $G={\rm SO}(n)$ and $\rho\colon M^{\sharp}\to M$ be the $G$-principal bundle of orthonormal oriented frames of $(M,\F)$ and let $\F^{\sharp}$ be the induced transversally parallelizable Riemannian foliation on $M^{\sharp}$. Denote, accordingly to Molino's theorem \ref{molinoth}, by  $W$ the basic manifold of $\F$ and by 
$F\hookrightarrow M^{\sharp}\to W$ the locally trivial fibration induced by the the leaf closures of $\F^{\sharp}$. 
Denote by $E^{\sharp}\to M^{\sharp}$ the pull-back of $E$ via $\rho$. 
In view of \cite{EKA} there always exist a $G$-bundle 
$\bar{E}\to \bar W$ and canonic isomorphisms 
$$
\psi^\sharp\colon C^{\infty}(E^{\sharp}/\F^{\sharp})\to C^{\infty}(\bar E)\,,\quad \psi\colon C^{\infty}(E/\F)\to C^{\infty}_G(\bar E)\,.
$$
The following result is proved in \cite{EKA}

\begin{theorem}[El Kacimi]\label{ekatheorem}
Let $D\colon C^{\infty}(E/\F)\to C^{\infty}(E/\F)$ be a linear strongly transversally elliptic basic differential operator.
Then there exists a $G$-invariant  strongly elliptic differential operator $\bar D\colon C^{\infty}(\bar E)\to C^{\infty}(\bar E)$ making the following diagram commutative  
\b\label{commutative}
\begin{array}[c]{ccc}
C^{\infty}(E/\F) &\xrightarrow{\quad \bar D\quad }&C^{\infty}(E/\F)\\
\Big\downarrow\scriptstyle{\psi}&&\Big\downarrow\scriptstyle{\psi}\\
C^{\infty}_G(\bar E)&\xrightarrow{\quad  D\quad }&C^{\infty}_G(\bar E)\,.
\end{array}
\e
 
\end{theorem}

\begin{proof}[Proof of theorem \ref{parabolic} in the linear case] 
In the linear case, the statement of theorem \ref{parabolic} easily follows from theorem \ref{ekatheorem} and lemma \ref{lemmaginv}. Indeed, theorem \ref{ekatheorem} implies that if $D$ is linear, then there exists $\bar D\colon C^{\infty}(\bar E)\to C^{\infty}(\bar E)$ as in the statement of theorem \ref{ekatheorem}. 
Let $\bar u_0=\psi (u_0)$. Then theorem \ref{classical existence} implies that the parabolic problem 
$$
\partial_t \bar u(t)=\bar D(\bar u_t)\,,\quad \bar u_{|t=0}=\bar u_0
$$
has a unique solution $\bar u\in C^{\infty}(\bar E\times [0,\infty))$. Since $\bar u_0$ and $\bar D$ are $G$-invariant, then lemma \ref{lemmaginv} ensures that $\bar u_t$ stays $G$-invariant for every $t$.  Hence we can write $\bar u_t=\psi(u_t)$ for some smooth curve $u$ in $C^{\infty}(E/\F)$ solving \eqref{parabolic}. The uniqueness of $u$ follows from the fact that $\psi$ is a isomorphism and the uniqueness of standard parabolic problems.
\end{proof}
The proof of theorem \ref{existence} in the nonlinear case works in the same way, but since El Kacimi's theorem \ref{ekatheorem} is proved in \cite{EKA} only for linear operators, we have to extend it to the quasilinear case. 
\begin{theorem}\label{ourEKA}
Let $D\colon C^{\infty}(E/\F)\to C^{\infty}(E/\F)$ be a quasilinear basic differential operator which is strongly transversally elliptic at $u\in C^{\infty}(E/\F)$.
Then there exists a $G$-invariant  differential operator $\bar D\colon C^{\infty}(\bar E)\to C^{\infty}(\bar E)$ which is strongly elliptic at $\psi(u)$ and makes the following diagram commutative  
\b\label{commutative}
\begin{array}[c]{ccc}
C^{\infty}(E/\F) &\xrightarrow{\quad \bar D\quad }&C^{\infty}(E/\F)\\
\Big\downarrow\scriptstyle{\psi}&&\Big\downarrow\scriptstyle{\psi}\\
C^{\infty}_G(\bar E)&\xrightarrow{\quad  D\quad }&C^{\infty}_G(\bar E)\,.
\end{array}
\e

\end{theorem}

\begin{proof}
Assume that $D$ has even degree $r$. 
In terms of jets $D$ can be written as $D=T\circ J_r$, where we still denote by $T\colon C^{\infty}(J^r(E/\F))\to C^\infty(E)$ the map induced by the natural morphism 
$$
T\colon J^r(E/\F)\to E
$$
(see the discussion at the end of section \ref{basicdiffoperfolaitedmanifolds}).
It is clear that $D$ is linear if and only if $T_{x}$ is linear for every $x\in M$.
Let $g_Q$ be the transverse metric of $\F$; then the total space $M^{\sharp}$ inherits a transversally oriented foliation $\F^{\sharp}$ which is ${\rm SO}(n)$-invariant.  
The connection $H^{\sharp}$ on $M^{\sharp}$ induced by the transverse Levi-Civita connection of $g_Q$ gives the splitting $TM^{\sharp}=H^{\sharp}\oplus V$, where $V$ is the vertical bundle. Since the induced bundle $L^{\sharp}$ is by definition a subundle of $H^{\sharp}$, then $Q^{\sharp}=TM^{\sharp}/L^{\sharp}$ inherits the splitting 
$$
Q^{\sharp}=H^b\oplus V^b\,,
$$
where $V^b$ is isomorphic to $V$ and $H^b=H^{\sharp}/L^{\sharp}$. Let 
$\pi^{\sharp}\colon E^{\sharp}\to M^{\sharp}$ be the pull-back of $E$ via $\rho$ and $\rho^{\sharp}\colon E^{\sharp}\to E$ the map induced by $\rho$, i.e. 
\begin{equation*}
\begin{array}[c]{ccc}
E^\sharp&\xrightarrow{\quad  \pi^{\sharp}\quad }&M^\sharp\\
\Big\downarrow\scriptstyle{\rho^\sharp}&&\Big\downarrow\scriptstyle{\rho}\\
E&\xrightarrow{\quad  \pi\quad }&M
\end{array}
\end{equation*}
We denote by $S^k(Q,E)^{\sharp}$ the pull-back of $S^k(Q,E)$ to $M^{\sharp}$. Then we easily get    
$$
S^k(Q,E)^{\sharp}\simeq S^k(Q^{\sharp}/V^b,E^{\sharp})\simeq S^k(H^b,E^{\sharp})\,.
$$
Now for every $k$ we can split $S^k(Q^{\sharp},E^{\sharp})$ as 
$$
S^k(Q^{\sharp},E^{\sharp})=\bigoplus_{i+j=k}S^{i,j}
$$
where 
$$
S^{i,j}:=S^i(H^b,E^{\sharp})\otimes S^j(V^b,E^{\sharp})\,.
$$ 
This fact allows as to lift the map $T$ to a map 
$ T^\sharp \colon \bigoplus_{k=0}^r S^k(Q^\sharp,E^{\sharp})\to E^{\sharp}$ making the following diagram commutative 

\begin{equation}\label{lift}
\begin{array}[c]{ccc}
\bigoplus_{k=0}^r S^k(Q^{\sharp},E^\sharp)&\xrightarrow{\quad  T^{\sharp}\quad }&E^\sharp\\
\Big\downarrow\scriptstyle{\rho_s^\sharp}&&\Big\downarrow\scriptstyle{\rho^{\sharp}}\\
\bigoplus_{k=0}^r S^k(Q,E)&\xrightarrow{\quad  T\quad }&E
\end{array}
\end{equation}
where $\rho^{\sharp}_s$ is induced by $\rho^{\sharp}$. The map $T^{\sharp}$ is defined as follows: \\
Let $x^\sharp\in M^\sharp $ and $x=\rho(x^\sharp)$; since 
$$
\rho^{\sharp}\colon E^{\sharp}_{x^\sharp}\to E_x
$$
is an isomorphism, if $\theta\in S^k_{x^\sharp}(H^b,E^{\sharp})$, then we can define  
$$
\tilde T_{x^\sharp}(\theta)=\left(\rho^{\sharp}\right)^{-1}_x(T_z(\rho^\sharp_{s,x}(\theta)))
$$
In this way we have a map $T^{\sharp}$ making the following diagram commutative 
$$
\begin{array}[c]{ccc}
\bigoplus_{k=0}^r S^k(H^b,E^\sharp)&\xrightarrow{\quad  T^{\sharp}\quad }&E^\sharp\\
\Big\downarrow\scriptstyle{\rho^\sharp_s}&&\Big\downarrow\scriptstyle{\rho^\sharp}\\
\bigoplus_{k=0}^r S^k(Q,E)&\xrightarrow{\quad  T\quad }&E
\end{array}
$$
and we extend $T^{\sharp}$ as the null map in $S^{i,j}$ whenever $j\neq 0$.  
Keeping in mind the isomorphism 
$$
\bigoplus_{k=0}^r S^k(Q^{\sharp},E^{\sharp})\simeq J^{r}(E^{\sharp}/\F^{\sharp})
$$ 
we can use the map $T^{\sharp}$ to define the partial differential operator $D^{\sharp}=T^{\sharp}\circ J_r^{\sharp}$ acting on $C^{\infty}(E^{\sharp}/\F^{\sharp})$. Now
$D^{\sharp}$ induces the genuine partial differential operator 
$\bar D\colon C^{\infty}(\bar E)\to C^{\infty}(\bar E)$ by defying $\bar D=\psi^{\sharp}D^{\sharp}(\psi^{\sharp})^{-1}$. By construction $\bar D$ is $G$-invariant and makes diagram \eqref{commutative} commutative.
Assume that $D$ is strongly transversally elliptic at $u$
and fix a complement $H$ of $L^{\sharp}$ in $H^{\sharp}$ in order to have the splitting 
$$
TM^{\sharp}=L^{\sharp}\oplus H\oplus V\,.
$$
We firstly show that $D^{\sharp}$ is $H^*$-strongly transversally elliptic at $u$. By differentiating the following commutative diagram at $u^{\sharp}$
\begin{equation*}
\begin{array}[c]{ccc}
C^{\infty}(J^{r}(E^{\sharp}/\F^{\sharp}))&\xrightarrow{\quad  T^{\sharp}\quad }&C^{\infty}(E^{\sharp})\\
\Big\downarrow\scriptstyle{\rho_s^\sharp}&&\Big\downarrow\scriptstyle{\rho^{\sharp}}\\
C^{\infty}(J^{r}(E/\F))&\xrightarrow{\quad  T\quad }&C^{\infty}(E)
\end{array}
\end{equation*}
we get 
\begin{equation*}
\begin{array}[c]{ccc}
C^{\infty}(J^{r}(E^{\sharp}/\F^{\sharp}))&\xrightarrow{\quad  T^{\sharp}_{*|J_{r}\left(u^{\sharp}\right)}\quad }&C^{\infty}(E^{\sharp})\\
\Big\downarrow\scriptstyle{\rho^\sharp_s}&&\Big\downarrow\scriptstyle{\rho^{\sharp}}\\
C^{\infty}(J^{r}(E /\F))&\xrightarrow{\quad  T_{*|J_{r}(u)}\quad }&C^{\infty}(E)
\end{array}
\end{equation*}
where $u=\rho^{\sharp}(u^{\sharp})$. Since $D_{*|u}$ is strongly transversally elliptic and 
$$
D_{*|u}=T_{*|J_{r}(\alpha)}\circ J_r\,,\quad 
D^{\sharp}_{*|u^{\sharp}}=T^{\sharp}_{*|J_{r}(u^{\sharp})}\circ J_r\,,
$$
then \cite[proposition 2.8.5]{EKA} implies that $D^{\sharp}_{*|\alpha^{\sharp}}$ is  $H^*$-strongly transversally elliptic. Although the induced $D^{\sharp}$ cannot be transversally strongly elliptic, 
we can correct it with some extra terms according to the construction described in \cite{EKA}. Let $\{X_1,\dots, X_N\}$ be a basis of the Lie algebra of SO$(n)$ and let $Q_j$ be the SO$(n)$-invariant differential operator on $\Gamma(E^{\sharp})$ defined by
$$
(Q_{j}\alpha)(x):=\frac{d}{dt}\alpha\left({\exp(tX_j)x}\right)_{|t=0}\,.
$$     
Then we set 
$$
Q:=(-1)^{r/2}\left( \sum_{i=1}^N Q_j \circ Q_j\right)^{r/2}\,.
$$
Clearly $Q$ is $V^*$-strongly elliptic and null in $C^{\infty}_G(E^{\sharp}/\F^{\sharp})$. 
Let $D':=D^{\sharp}+Q$ and $\bar D'\colon C^{\infty}(\bar E)\to C^{\infty}(\bar E)$ be the induced operator.  Since $\bar D_{|*\psi(u)}'$ is the operator induced by 
 $D^{\sharp}_{*|u^{\sharp}}+Q$, \cite[Proposition 2.8.6]{EKA} implies that $\bar D_{|*\psi(u)}'$ is strongly elliptic and the claim follows.  
\end{proof}

\section{The transverse Ricci flow}
In this section we give a  proof of the well-posedness of the of the transverse Ricci flow based on theorem \ref{existence}: the short-time existence is treated in the spirit of \cite{deTurck}, while the uniqueness is obtained with the energy approach of \cite{Kotschwar}.

\medskip
We briefly recall the definition of the flow introduced in \cite{Min-Oo}.  
Let $M$ be a compact $(m+n)$-dimensional manifold equipped with an $n$-codimensionial Riemannian foliation with tangent bundle $L$. We denote as usual by $\pi\colon Q\to M$ the normal bundle $TM/L$ and by $g_Q$ the transverse metric. We also assume $\F$ {\em homologically oriented} by a form $\chi$ (see definition \ref{homorientable}). 
Let $g$ be a bundle-like metric on $(M,\F)$ inducing $g_Q$ and let $\sigma: Q \to L^{\perp}$ be the map which assigns to each $[v] \in Q$ the component of $v$ orthogonal to $L$ with respect to $g$. We denote by $\nabla$ the transverse Levi-Civita connection \eqref{TLC} induced by $g_Q$ and by  $R_{Q}$ its curvature  adopting the sign convention 
$$
R_{Q}(X,Y) = \nabla_X\nabla_Y - \nabla_Y\nabla_X - \nabla_{[X,Y]}\,.
$$
The {\em transverse Ricci curvature} of $g_Q$ is then the basic
tensor ${\rm Rc}_Q \in \Gamma(S^2 Q^*)$ defined by 
$$
{\rm Rc}_Q(V,W)=g_Q(R_{Q}(\sigma(V),\sigma(e_i))e_i,W)\,,
$$
where $\{e_i\}_{i=1}^n$ is a $g_Q$-orthonormal frame of $Q$. We further denote by $s_Q$ the transverse scalar curvature of $g_Q$ which is defined by 

\begin{equation}\label{sQ}
s_Q:=\sum_{k=1}^n{\rm Rc}_Q(e_k,e_k)\,.
\end{equation}

The flow introduced in \cite{Min-Oo} is  the flow $g_Q(t)$ of transverse Riemannian metrics governed by the equation
$\partial_tg_{Q}(t)=-2{\rm Rc}_Q(t)$, 
where ${\rm Rc}_Q(t)$ is the transverse Ricci curvature induced by the transverse metric $g_Q(t)$. The main result of this section is the following
\begin{theorem}
Let $(M,\F)$ be a compact manifold equipped with a {\em homologically oriented} Riemannian foliation and let $\tilde g_Q$ be a smooth holonomy invariant metric on the quotient bundle $Q$. Then the evolution equation
\b\label{TRF2}
\partial_tg_Q(t)=-2{\rm Rc}_Q(t)\,,\quad g_{Q}(0)=\tilde g_{Q}
\e
has a unique short-time solution. 
\end{theorem}

As in the non-foliated setting the evolution equation \eqref{TRF2} cannot be parabolic because it is invariant by diffeomorphisms preserving $\F$.  However it can be made parabolic by the de Turck-like trick we are going to describe.

We regard ${\rm Rc}_Q$ as an operator on the space of transverse Riemannian metrics on $(M,\F)$, which is open in the space of basic sections of $S^2 Q^*$. Let $\{x^1,\dots,x^m,y^1,\dots,y^n\}$ be a foliated coordinate system. A local frame of $Q$ is obtained by taking 
$V_i=\pi(\partial_{y^i})$ for $i=1,\ldots,n$. Hence locally $\nabla$ is described by the functions $\Gamma_{ij}^k$ defined by
\begin{equation}
\label{Chris}
\nabla_{\partial_{y^i}}V_j=\Gamma_{ij}^k V_k\,.
\end{equation}
Note that this is equivalent to say
$\nabla_{\partial_{y^i}}\pi(\partial_{y^j})=\Gamma_{ij}^k \pi(\partial_{y^k})\,.$
Let $g_{ij}:=g_Q(V_i,V_j)$ and let $g^{rs}$ be the components of the inverse matrix of $(g_{ij})$. Then 
$$
\begin{aligned}
{\rm Rc}_Q(V_i,V_j) = g^{kl}\, g(R_Q(\partial_{y^i},\partial_{y^k})V_l,V_j)
=&\, g^{kl}\partial_{y^i}(\Gamma_{kl}^r)g_{rj}- g^{kl}\partial_{y^k}(\Gamma_{il}^r)g_{rj}+{\rm l.o.t.}
\end{aligned}
$$ 
Once a background transverse metric $\hat g_Q$ is fixed, every other transverse metric $g_Q$ induces the basic vector field 
\begin{equation}
\label{campo deTurck}
X= g^{ij}(\hat{\Gamma}^{k}_{ij}-\Gamma_{ij}^k)\,\partial_{y^k}
\,,
\end{equation}
where the functions $\hat{\Gamma}^{k}_{ij}$ are defined by $\hat{\nabla}_{\partial_{y^i}}V_j=\hat{\Gamma}_{ij}^k V_k$ and $\hat{\nabla}$ is the transverse Levi-Civita connection of $\hat{g}_Q$ on $Q$.
Then
$$
(\mathcal{L}_{X}g_Q)_{ij}=-g^{kr} \partial_{y^i} (\Gamma_{kr}^l)g_{lj}-g^{kr}\partial_{y^j}(\Gamma_{kr}^l)g_{li}+{\rm l.o.t.}
$$
Now from \eqref{Chris}, \eqref{free} and the definition of the frame $V_i$, it easily follows
$$
\Gamma_{ij}^k=\frac{1}{2}\left(\partial_{y^i}(g_{jk})+\partial_{y^i}(g_{ik})-\partial_{y^k}(g_{ij})\right)\,.
$$
Finally we have 
$$
(-2{\rm Rc}_Q-\mathcal{L}_{X}g_Q)(V_i,V_j) = g^{kl} \partial_{y^k} \partial_{y^l}(g_{ij}) + {\rm l.o.t.} =\Delta_B g_{ij} +{\rm l.o.t.}\,,
$$
where $\Delta_B$ is the basic Laplacian of $g_Q$.
This shows that the operator 
$$
g_Q\mapsto -2{\rm Rc}_Q-\mathcal{L}_{X}g_Q
$$ 
is strongly  transversally elliptic on an open subset of $C^\infty(S^2 Q^*/\F)$.
Thus we have the following proposition which is now a consequence of theorem \ref{parabolic}
\begin{prop}
\label{Exmod}
Let $(M,\F)$ be a compact manifold equipped with a transversally orientable  Riemannian foliation with transverse metric $\tilde g_Q$. There exists a $T>0$ and a smooth one-parameter family of transverse metrics $g_Q(t) \in C^\infty(S^2 Q^*/\F)$, $t \in [0,T)$, solving
\begin{equation}
\label{modTRF}
\partial_tg_Q(t)=-2{\rm Rc}_Q(t)-\mathcal{L}_{X_t}g_Q(t)\,,\quad g_Q(0)=\tilde g_Q\,, 
\end{equation}
where $X$ is given by \eqref{campo deTurck}. 
\end{prop}

About the existence of a solution to \eqref{TRF2},
we reconstruct a solution of the transverse Ricci flow \eqref{TRF2} from the modified flow \eqref{modTRF}.
In order to do this, we firstly integrate the time-dependent vector field $X_t$ to a 1-parameter group $\phi_t$ of diffeomorphisms of $M$ and observe that by definition \eqref{campo deTurck} these diffeomorphisms preserve the foliation, i.e. $(\phi_t)_*(L_x) = L_{\phi_t(x)}$ for any $x \in M$ and any $t$. Hence if $g_Q(t)$ is a solution of \eqref{modTRF},  we can define $\phi_t^*(g_Q)$ by means of
$$
\phi_t^*(g_Q)(V,W)=g_Q(\pi(\phi_{t*}\tilde V),\pi((\phi_{t*}\tilde W))\,,
$$
where $V,W \in Q$ and $\pi(\tilde V)=V$ and $\pi(\tilde W)=W$. It is immediate to verify that $\phi_t^*(g_Q)$ is a solution of the original transverse Ricci flow \eqref{TRF2}.

In order to prove the uniqueness of the transverse Ricci flow, we adapt the argument of \cite{Kotschwar} to the compact  foliated case. 
Assume that $g_Q$ and $\bar g_Q$ are two solutions in the interval $[0,T]$ of the transverse Ricci-flow with the same initial value ${\tilde g_Q}$. We denote by $\N$ and $\bar \N$ the induced transverse Levi-Civita connections and by $R^{\nabla}$ and $R^{\bar \nabla}$ the transverse curvature tensors. Let $\{e_1, \ldots, e_n\}$ and $\{\bar e_1,\ldots, \bar e_n\}$ be two local frames of $Q$ orthonormal with respect to $g_Q$ and $\bar g_Q$ respectively.
Then we define the following smooth tensors on $M\times [0,T]$
$$
h=g_Q-\bar g_Q\,,\quad A=\nabla-\bar \nabla\,,\quad S=R^{\nabla}-R^{\bar \nabla}\,. 
$$
and their norms with respect to $g_Q$
$$
\begin{aligned}
&|h|^2=|g_Q-\bar g_Q|_{g_Q}^2=\sum_{i,j=1}^n(\delta_{ij}-\bar g_{Q}(e_i,e_j))^2 \\
&|A|^2=|\nabla-\bar \nabla|_{g_Q}^2=\sum_{i,j,k=1}^n\left(g_Q(\nabla_{\bar e_i}e_j,e_k)-\bar g_Q(\bar \nabla_{\bar e_i}e_j,e_k)\right)^2\\
&|S|^2=|R^{\nabla}-R^{\bar \nabla}|_{g_Q}^2=\sum_{i,j,k,l=1}^n \left(g_Q(R^{\nabla}(\bar e_i,\bar e_j)e_k,e_l)-\bar g_Q(R^{\bar \nabla}(\bar e_i,\bar e_j)e_k,e_l)\right)^2
\end{aligned}
$$
and consider the function $\mathcal{E}\colon [0,T]\to \R^+$
$$
\mathcal{E}(t)=\int_M \left( |h|^2+  |A|^2 + |S|^2\right) \, \chi\wedge d\mu\,,
$$
where $d\mu$ is the time dependent family of transverse volume forms induced by $g_Q(t)$ and the transverse orientation. 
To conclude that indeed $\mathcal{E}(t)$ vanishes identically on $[0,T]$, we only need to prove the following proposition and apply Gronwall's lemma. This is analogous to proposition 7 in \cite{Kotschwar}
\begin{prop}\label{propkotchwar}
There exists a constant $C_0$ 
depending on $n$ and an upper bound on $R^{\nabla}$ and $R^{\bar \nabla}$ and their first derivatives, 
such that $\mathcal{E}'(t) \leq C_0 \mathcal{E}(t)$, for all $t \in [0,T]$.
\end{prop}

The necessary ingredients are contained in the following lemmas 

\begin{lemma}
\label{Kotlemma}
The following estimates hold 
\begin{eqnarray}
&& |\partial_t h|\leq C|S|\,,\\
&& \label{partialA}|\partial_t A |\leq C(|\bar g_Q^{-1}| |\bar \nabla R^{\bar \nabla}| |h|+|R^{\bar \nabla}| |A|+ |\nabla S|)\,.
\end{eqnarray}
Moreover, if $U=g^{ab}\nabla_bR^{\nabla}-\bar{g}^{ab} \bar\nabla_b  R^{\bar\nabla} $, then 
\b\label{partialS}
|\partial_t S-\Delta S-{\rm div}\,U|\leq C\left(|\bar g^{-1}_Q| |\bar \nabla R^{\bar{\nabla}}| |A|+ |\bar g^{-1}| |R^{\bar{\nabla}}|^2 |h|+(|R_Q|+|R^{\bar{\nabla}}|)|S| \right)
\e
and 
\b\label{normofU}
|U|\leq C(|\bar g^{-1}| |\bar \nabla R^{\bar{\nabla}}| |h|+|A| |R^{\bar{\nabla}}|)\,. 
\e 
where $({\rm div}\, U)^l_{ijk}=\nabla_a U^{al}_{ijk}$ and in all the inequalities the constant $C$ depends only on the codimension of the foliation.
\end{lemma}

\begin{proof}
The estimates are proved in \cite{Kotschwar} for the Ricci flow in the non-foliated case. Since all the estimates in \cite{Kotschwar} are local and a solution of the transverse Ricci flow can be regarded as a collection of solutions to the Ricci flow on open sets in $\mathbb R^n$, the claim follows.     
\end{proof}

\begin{lemma}
\label{uniflemma}
The metrics $ g_Q(t)$, $\bar  g_Q(t)$, $\tilde g_Q$ are all uniformly equivalent in $[0,T]$.
\end{lemma}
\begin{proof}
The statement follows from \cite[Theorem 14.1]{positive}
\end{proof}
\begin{cor}
The following estimates hold 
\begin{eqnarray}
&&\label{partialA2} |\partial_t A |\leq C( |h|+ |A|+ |\nabla S|)\,,\\
&&\label{partialS} |\partial_t S-\Delta S-{\rm div}\,U|\leq C\left(|A|+  |h|+|S| \right)\,,\\
&&\label{normU} |U|\leq C( |h|+|A|)\,,
\end{eqnarray}
where the constants depend on $n$, $T$ and an upper bound of the curvatures and its first derivatives. 
\end{cor}
\begin{proof}
The inequalities are obtained by combining lemma \ref{Kotlemma} and lemma \ref{uniflemma}.
\end{proof}

\begin{proof}[Proof of Proposition $\ref{propkotchwar}$]
Let us define 
$$
\mathcal H=\int_M|h|^2\,\chi\wedge d\mu\,,\quad \mathcal I=\int_M |A|^2\,\chi\wedge d\mu, \quad \mathcal J=\int_M |\nabla S|^2\,\chi\wedge d\mu
\,,\quad \mathcal G=	\int_M |S|^2\,\chi\wedge d\mu 
$$
so that 
$$
\mathcal E=\mathcal H+\mathcal I+ \mathcal G\,. 
$$
Then 
$$
\begin{aligned}
\partial_t \mathcal H\leq C\mathcal H+\int_M \langle \partial_t h,h\rangle\,\chi\wedge d\mu \leq C\mathcal H+
\int_M C |S| |h|\,\chi\wedge d\mu
\end{aligned}
$$
and using 
$$
 C |S| |h|\leq C |h|^2+C|S|^2
$$
we get 
\b
\begin{aligned}
\partial_t \mathcal H\leq C\mathcal H +C \mathcal G\,.
\end{aligned}
\e
Moreover,
$$
\begin{aligned}
\partial_t \mathcal I\,&\leq C\mathcal I+ \int_M \langle \partial_t A ,A\rangle \chi\wedge d\mu \\
& \leq C\mathcal I+ \int_M C \left(|\tilde g^{-1}| |\bar  \nabla \bar R| |h|+ |\bar R| |A| +|\nabla S|\right) |A| \chi\wedge d\mu\\
& \leq C\mathcal I+ \int_M C  |h|  |A|+ |\nabla S| |A| \chi\wedge d\mu\\
\end{aligned}
$$
and using 
$$
 C  |h|  |A|+   |\nabla S| |A|\leq C |h|^2+C|A|^2+|\nabla S|^2
$$
we get 
\b
\partial_t \mathcal I\leq C\mathcal H+\mathcal J+ C\mathcal I\,.
\e 
Moreover, 
$$
\begin{aligned}
\partial_t\mathcal G\,&\leq C \mathcal G+\int_M 2\langle \partial_t S,S  \rangle \chi\wedge d\mu \\
& \leq C \mathcal G+ \int_M (2\langle \Delta S+{\rm div} U,S \rangle +C|\bar g^{-1}| |\bar \nabla \bar R| |A| |S|+ C |\bar g^{-1}| |\bar R|^2 |h|  |S|
+(|R|+|\bar R|)|S|^2) \chi\wedge d\mu\\
& \leq C \mathcal G+ \int_M (2\langle \Delta S+{\rm div} U,S \rangle +C |A| |S|+ C |h|  |S|
+C|S|^2) \chi\wedge d\mu\\
& \leq C \mathcal G+ C \mathcal{I}+ C \mathcal H+ \int_M 2\langle \Delta S+{\rm div} U,S \rangle \chi\wedge d\mu\,.
\end{aligned}
$$
Now 
$$
\begin{aligned}
 \int_M 2\langle \Delta S+{\rm div} U,S \rangle \chi\wedge d\mu\leq -2 \mathcal J+ 2\int_M |\nabla S| |U| \,\chi\wedge d\mu\leq - \mathcal J+ 3\int_M |U|^2 \,\chi\wedge d\mu
\end{aligned}
$$
and then we get
$$
\int_M 2\langle \Delta S+{\rm div} U,S \rangle \chi\wedge d\mu\leq - \mathcal J+ C\mathcal H+C \mathcal I\,.
$$
Therefore
$$
\partial_t\mathcal G\,\leq  C \mathcal G+ C \mathcal{I}+ C\mathcal H - \mathcal J\,,
$$
which implies
$$
\partial_t\mathcal{E}\leq C\mathcal E
$$
and the statement follows. 
\end{proof}

\section{The transverse K\"ahler-Ricci flow}\label{kahler1}
This section is about the generalization of the K\"ahler-Ricci flow to transverse geometry. 
The K\"ahler-Ricci flow is a powerful tool for studying K\"ahler manifolds which was introduced by Cao in \cite{cao}. 
In \cite{SWZ} Smoczyk, Wang and Zhang generalized the flow to Sasakian manifolds proving a \lq \lq Sasakian version'' of Cao's theorem. 

\medskip 
 A {\em K\"ahler foliation} is by definition a foliation $\F$ provided with a transverse K\"ahler structure (see e.g. \cite{BGlibro} and \cite{Futaki}). Such a structure can be regarded as a pair  $(g_Q,J)$ of tensors on the normal bundle of the foliation $Q$, where 
$g_Q$ is a transverse metric making the foliation Riemannian and $J$ is an endomorphism of $Q$ 
satisfying $J^2=-{\rm Id}$, $g_Q(J\cdot,J\cdot)=g_{Q}(\cdot,\cdot)$ and an integrability condition.  The pair $(g_Q,J)$ induces a closed basic $2$-form $\omega$  on $M$ defined as the pull-back of $g_Q(J\cdot,\cdot)$. We refer to 
$\omega$ as to the {\em fundamental form} of the transverse K\"ahler structure. 
The transverse complex structure $J$ induces a natural splitting of the space $\Omega^r_B(M,\C)$ of complex basic $r$-forms on $M$ into $\Omega^r_B(M,\C)=\oplus_{p+q=r}\Omega_B^{p,q}$ and the restriction of $d_B$ to basic complex $(p,q)$-forms splits accordingly as $d_B=\partial_B+\bar\partial_B$. As in the non-foliated case $\partial_B^2=\bar\partial_B^2=0$ and these operators define some cohomology groups (see e.g. \cite{EKA,BGlibro} for details).  From the local point of view it is useful to recall that we can always find  coordinates $\{x^1,\dots,x^m,z^1,\dots, z^n\}$ taking values in $\R^m\times \C^n$, such that $\{x^1,\dots, x^m\}$ are coordinates on the leaves and  $\{V_k:=\pi(\partial_{z^k})\}$ is a local $(1,0)$-frame of $Q$. Such coordinates are usually called {\em complex foliated}.

 From now on we assume $M$ {\em compact} and $\F$ {\em homologically oriented} by a form $\chi$ on $M$. The existence of $\chi$  allows us to generalize many results about K\"ahler manifolds to the non-foliated case. For instance, El Kacimi proved in \cite{EKA} a foliated version of the $\partial\bar\partial$-lemma (called the 
 $\partial_B\bar\partial_B$-lemma) and  gave a generalization of the Calabi-Yau theorem. 
Indeed, accordingly to the non-foliated case,  it is defined the transverse Ricci form of $(\omega,J)$ as a closed basic form $\rho_B$ on $M$ obtained as pull-back of 
${\rm Rc}_{Q}(J\cdot,\cdot)$ to $M$.
Such a form locally writes as $\rho_B=-i\partial_B\bar\partial_B\log\det(g_{k\bar r})$, 
where we locally write $g_Q=g_{r\bar s}dz^rd\bar z^s$, and 
allows us to define the {\em basic first Chern class} as 
$$
c^1_B(M):=\frac{1}{2\pi}[\rho_B]\in H^{2}_B(M)\,.
$$
We recall the following 

\begin{theorem}[El Kacimi \cite{EKA}]
For every $\beta$ representing $c_B^1(M)$ there exists a unique K\"ahler form in the same basic cohomology class as $\omega$ whose transverse Ricci form is $2\pi\beta$.

\end{theorem}
\medskip 
Here we want to study the transverse version of the K\"ahler-Ricci flow for K\"ahler foliations. Let $M$ be a compact manifold equipped with an initial K\"ahler foliation $(\F,\tilde g_{Q},J)$ and consider 
 the transverse Ricci flow
\b
\label{TKRF} \partial_tg_Q(t)=-{\rm Rc}_Q(t)\,,\quad 
g_Q(0)=\tilde g_{Q}\,.
\e

In this case we can prove the following two results 
\begin{theorem}\label{timedomain}
There exists a unique smooth family of transverse K\"ahler metrics $g_Q(t)$, defined for $t\in [0,T)$,  such that $g_Q(t)$ solves \eqref{TKRF} where
$$
T=\sup_{t>0}\, \left\{[\tilde \omega]_B-2\pi t\,c_B^1(M,J)>0\right\}\,,
$$
and $\tilde \omega$ is the fundamental form of $\tilde g_{Q}$. 
Moreover  if $c_B^1(M,J)=0$, then  $g_Q(t)$ converges to a transversally Ricci-flat metric. 
\end{theorem}
In the statement above, when we write that a class $\gamma\in H^{1,1}_B(M)$ is positive we mean that there exists a form $\kappa \in \gamma$ which is the fundamental form of a transverse Hermitian metric on $(M,\F,J)$.

\begin{theorem}\label{c1<0}
If $c_1^B(M,J)=\nu[\tilde \omega]_B$ with $\nu<0$, then there exists a unique smooth family of transverse K\"ahler metrics $g_Q(t)$ defined for $t\in [0,\infty)$,  whose fundamental form $\omega_t$ solves 
\b\label{mTKRF}
\partial_t\omega_t=-\rho_B(\omega_t)-\nu \omega_t,\quad 
\omega_{|t=0}=\tilde \omega\,,
\e
and $g_Q(t)$ converges to a  transversally K\"ahler-Einstein metric. 
\end{theorem}

The short-time existence and the uniqueness for the solutions to the transverse K\"ahler-Ricci flow will be obtained by using theorem \ref{parabolic}, while the long time behaviour will be studied working as in K\"ahler geometry. For the long time existence we follow the description in \cite{SW} omitting those computations which totally agree to the non-foliated case.


\subsection{Some known results in open K\"ahler Manifolds}\label{kahler}
Since the transverse K\"ahler-Ricci flow looks locally as a collection of K\"ahler-Ricci flows on open sets of $\C^n$, we can  use the local estimates for the K\"ahler-Ricci flow to study the transverse case. 
In this subsection we recall some results involving K\"ahler structures on non-compact K\"ahler manifolds. The first of them is the following easy-to-prove lemma of linear algebra  

\begin{lemma}\label{LAlemma}
Let $V$ be  an $n$-dimensional complex vector space and let $\omega_1$ and $\omega_2$ be two positive $(1,1)$-forms and let $A$  and $B$ two positive constants. 

\begin{enumerate}

\item[i.] If ${\rm tr}_{\omega_2}\omega_1+A\log\frac{\omega_2^n} {\omega_1^n}\leq B,$
then there exists a constant $C>0$ depending only on $A$, $B$ and $n$ such that ${\rm tr}_{\omega_1}\omega_2\leq C\,;$

\vspace{0.1cm}

\item[ii.] Assume $\omega_2\leq A\omega_1$ and $\omega_1^n\leq B\omega_2^n$, then there exists a constant $C>0$ depending only on $A$, $B$ and $n$ such that  $\omega_1\leq C \omega_2\,.$
\end{enumerate}
\end{lemma}

\medskip 
Let us consider now a K\"ahler manifold $(X,\tilde \omega)$  and let  $\omega_t$, $t\in [0,T]$, be a solution to the normalised K\"ahler-Ricci flow 
\b\label{KRF}
\partial_t\omega_t=-{\rm Ric}(\omega_t)-\nu\omega_t \,,\quad \omega_{|t=0}=\tilde \omega
\e
where $\nu$ is a non-negative real constant. The next lemma can be for instance easily deduced from theorem 2.2 and corollary 2.3 in \cite{SW}. Here and throughout this subsection the symbol $\Delta_t$ will stand for the {\em complex Laplacian} of the form $\omega_t$, i.e. $\Delta_t f= g_t^{\bar j i}\partial_i\partial_{\bar j}f$, where $f\in C^{\infty}(M)$. 

\begin{lemma}\label{lemmas}
Let $s_t$ be the scalar curvature of $\omega_t$, then $(\partial_t-\Delta_t)e^{\nu t}(s_t+\nu n)\geq 0$. \\
Moreover, assume that there exists a uniform constant $C$ such that $s_t\geq -\nu n-Ce^{-\nu t}$, then 
\begin{itemize}
\item if $\nu=0$, then $\omega^n_t\leq e^{Ct}\tilde \omega^n;$
\vspace{0.1cm}
\item if $\nu=1$, then $\omega^n_t\leq e^{C(1-e^{-t})}\tilde \omega^n$.     
\end{itemize}
\end{lemma}

Now we recall the following results involving the K\"ahler-Ricci flow (for the proofs we still refer to \cite{SW})

\begin{theorem}\label{thkahler}
Assume that there exists a uniform costant $C$ such that 
$\frac{1}{C}\tilde \omega\leq \omega_t\leq C\tilde \omega\,.$
Then any point $x\in X$ has a neighborhood $U$ where the $C^{\infty}$ norm of $\omega$ with respect to $\tilde \omega$ is uniformly bounded.
\end{theorem}

\begin{lemma}\label{estimate}
Let $\kappa$ be a K\"ahler structure on $X$ having bisectional curvature bounded from below, then there exists a uniform constant $C$ such that 
$$
(\partial_t-\Delta_t)\log{\rm tr}_{\kappa}\omega_t\leq C\,{\rm tr}_{\omega_t}\kappa-\nu\,.
$$
\end{lemma}

Let us consider now on  $(X,\tilde \omega)$ a solution  $\varphi$ to the parabolic Monge-Amp\`ere equation 
$$
\partial_t\varphi=\log\frac{(\tilde \omega+i\partial\bar\partial\varphi)^n}{\tilde \omega^n}-\varphi_t\,,\quad \tilde \omega+i\partial\bar\partial\varphi>0\,,\quad \varphi(0)=0
$$
defined in $X\times [0,\infty)$. 

\begin{lemma}\label{phiK}
Assume $\|\partial_t\varphi_t\|_{C^0}\leq Ce^{-t}$ for a uniform constant $C$. Then 
\begin{enumerate}
\item[1.] there exists a smooth map $\varphi_{\infty}$ on $X$ such that 
such that $\|\varphi_t-\varphi_{\infty}\|_{C^{0}}\leq C e^{-t}$;

\vspace{0.1cm}
\item[2.] $\frac{1}{C'}\tilde \omega^n\leq (\tilde \omega+i\partial\bar\partial\varphi)^n\leq C'\tilde \omega^n $ for a uniform constant $C'$. 
\end{enumerate}
\end{lemma}

\subsection{A maximum principle in foliated manifolds}
Here we prove a general maximum principle involving basic functions on compact manifolds foliated by Riemannian foliations. The result can be seen as an extension of \cite[Proposition 5.1]{zedda} to the foliated non-Sasakian case. 

By a {\em smooth family of linear basic partial differential operators} $\{E\}_{t\in [0,\epsilon)}$ we mean a smooth family of linear basic differential operators $E(\cdot,t)\colon C^{\infty}_B(M)\to C^{\infty}_B(M)$ whose coefficients depend smoothly on $t$. 

\begin{prop}[Maximum principle for basic maps]\label{maximum}
Let $(M,\F,g_Q,J)$ be a compact manifold with a K\"ahler foliation.  Let  $\{E\}_{t\in [0,\epsilon)}$ be a smooth family of linear basic partial differential operators such that $E(\cdot,t)$ is transversally strongly elliptic for every $t\in [0,\epsilon)$ and satisfies 
\b\label{E}
E(h(x,t),t)\leq 0
\e
whenever $h\in C^{\infty}_B(M\times[0,\epsilon))$ is such that $i\partial_B\bar \partial_Bh(x,t)\leq 0$.
Then if $h\in C^{\infty}_B(M\times [0,\epsilon),\mathbb{R})$ satisfies
$$
\partial_th(x,t)-E(h(x,t),t)\leq 0, 
$$
we have 
$$
\sup_{(x,t)\in M\times [0,\epsilon)} h(x,t)\leq \sup_{x\in M} h(x,0).
$$
\end{prop}
\begin{proof}
Fix $\epsilon_0\in (0,\epsilon)$ and let $h_\lambda\colon M\times [0,\epsilon_0]\to \R$ be 
$h_\lambda(x,t)=h(x,t)-\lambda t$. Assume that $h_\lambda$ achieves its global maximum at $(x_0,t_0)$ and assume by contradiction 
$t_0\neq 0$. Then 
$$
\partial_th_\lambda(x_0,t_0)\geq 0\,,\quad i\partial_B\bar\partial_B h_\lambda(x_0,t_0)\leq 0\,.
$$ 
Therefore condition \eqref{E} implies $E(h_\lambda(x_0,t_0),t_0)\leq 0$ and then
$$
\partial_th_\lambda(x_0,t_0)-E(h_\lambda(x_0,t_0),t_0)\geq 0. 
$$
Since $\partial_th_\lambda=\partial_th-\lambda$ and $E(h_\lambda(x,t),t)=E(h(x,t),t)$, we have 
$$
0\leq \partial_th(x_0,t_0)-E(h(x_0,t_0),t_0)-\lambda \leq -\lambda,
$$
which is a contradiction. Therefore $h_\lambda$ achieves its global maximum at a point $(x_0,0)$ and  
$$
\sup_{M\times [0,\epsilon_0]} h\leq \sup_{M\times [0,\epsilon_0]} h_\lambda+\lambda\epsilon_0\leq \sup_{x\in M} h(x,0)+\lambda\epsilon_0.
$$
Since the above inequality holds for every $\epsilon_0\in (0,\epsilon)$ and $\lambda>0$, the claim follows. 
\end{proof} 

\begin{cor}\label{maximumprinciple}
Let $(M,\mathcal F,g_Q(t),J)$ be a manifold with a family of K\"ahler foliations. 
Let $h\in C^{\infty}_B(M\times[0,T))$ which is basic for very $t$. Assume  
$$
(\partial_t-\Delta_{B,t})h_t\leq 0\,,
$$
where $\Delta_{B,t}$ is the basic Laplacian operator of $g_Q(t)$, 
then  
$$
\sup_{M\times [0,T)} h\leq \max_{M} h_0\,. 
$$
\end{cor}

\subsection{Proof of theorem \ref{timedomain}}
In this subsection we prove theorem  $\ref{timedomain}$. 

Every transverse volume form $\Omega$ on a manifold $M$ foliated by a K\"ahler foliation can be written as  
$$
\Omega=(i)^nf\,dz^1\wedge d\bar z^1\wedge \cdots \wedge dz^n\wedge d\bar z^n,
$$
where the map $f$ depends only on the transverse complex coordinates. Then when we write $\log \Omega$, we mean $\log f$. 
\begin{proof}[Proof of theorem $\ref{timedomain}$]
First of all we show that \eqref{KRF} has a unique transversally K\"ahler maximal solution $g_{Q}(t)$ defined in $M\times [0,T_{\max})$, where $T_{\max} \leq T$.

Let $T'<T$ and consider a transversally K\"ahler form $\beta$ such that 
$$
[\beta]_B=[\tilde \omega]_B-2\pi T'\,c_B^1(M,J)\,.
$$ 
Then we define  
$$
\hat \omega_t=\frac{1}{T'}((T'-t)\tilde \omega+t\beta)
$$
for $t\in [0,T')$ and we consider the scalar flow 
\b\label{parabolicMA}
\partial_t\varphi_t=\log\frac{(\hat\omega_t+i \partial_B\bar \partial_B \varphi_t)^n}{\Omega}\,,\qquad  \hat\omega_t+i\partial_B\bar\partial_B\varphi_t>0\,,\qquad \varphi_{|t=0}=0
\e
where  $\varphi_t$ is smooth family of basic functions and $\Omega$ is a transverse volume form satisfying 
$$
i\partial_B\bar \partial_B\log\Omega=\frac{1}{T'}(\beta-\tilde \omega). 
$$
Since \eqref{parabolicMA} is transversally parabolic, theorem \ref{parabolic} implies that it 
has a unique maximal short time solution $\varphi\in C^{\infty}_B(M\times [0,T_{\max}))$. Moreover the curve of metrics corresponding to the path of fundamental forms 
$$
\omega_t=\hat{\omega}_t+i\partial_B\bar \partial_B\varphi_t
$$ 
solves \eqref{TKRF} and the transverse $\partial_B\bar\partial_B$-lemma for K\"ahler foliations implies that every solution to \eqref{TKRF} induces a solution to \eqref{parabolicMA}. This implies the existence of a maximal solution $g_{Q}$ defined in $M\times [0,T_{\max})$. Since 
$d/dt[\omega(t)]_B=-(2\pi)\,c_1^B(M)$, we necessarily have $T_{\max}\leq T$. 

\medskip 
Next we study the long time behavior of the maximal solution $\omega_t$. 
Assume by contradiction $T_{\max}<T$ and for a fixed $T'$ such that $T_{\max}<T'<T$ and define $\hat \omega_t$ as above. Note that with our assumptions $T_{\max}$  is necessarily finite. 
 Then we can write $\omega_t=\hat \omega_t+i\partial_B\bar\partial_B \varphi_t$, where $\varphi$ solves \eqref{parabolicMA}.  In order to apply theorem \ref{thkahler}, we have to show that there exists a uniform constant $C$ such that $\frac{1}{C}\tilde \omega\leq \omega_t\leq C\tilde \omega$. That is equivalent to require $\frac{1}{C}\leq {\rm tr}_{\tilde \omega}\, \omega_t\leq C$ and it can be proved by providing some a priori uniform estimates involving $\varphi$.  

\medskip
\noindent $\bullet$ {\em $\|\varphi_t\|_{C^0}$ is uniformly bounded in $[0,T_{\max})$.}
Keeping in mind that $\varphi$ is a solution to  \eqref{parabolicMA}, it is not difficult to show that 
$\partial_t(\varphi_t-At)$ is negative for a constant $A$ sufficiently large and the maximum principle implies that $\partial_t(\varphi_t-At)$ achieves its maximum at $t=0$. Therefore $\varphi_t\leq T_{\max}A$. A similar argument yields a lower bound for $\varphi$.  

\medskip
\noindent $\bullet$ {\em $\|\partial_t\varphi_t\|_{C^0}$ is uniformly bounded in $[0,T_{\max})$.}
This is equivalent to $\frac{1}{C_1}\tilde \omega^n\leq  \omega_t^n\leq C_1\tilde \omega^n$, for a uniform constant $C_1$. Keeping in mind that the basic scalar curvature $s_B(t)$ and the basic Laplacian operator $\Delta_{B,t}$ of $g_Q(t)$ are locally the scalar curvature and the Laplacian of the K\"ahler base manifold $X$, then 
lemma \ref{lemmas} implies 
$(\partial_t-\Delta_{B,t})(e^{\nu t}s_B(t))\geq 0$. Therefore corollary \ref{maximumprinciple} implies  
$s_B(t)\geq -\nu n-C_2e^{-\nu t}$ for  a uniform constant $C_2$ and the second part of lemma \ref{lemmas} together with the compactness of $M$ implies $\omega_t^{n}\leq C_1\tilde \omega^n$ for a constant $C_1$. For the lower bound we have 
$$
(\partial_t-\Delta_{B,t})((T'-t)\partial_t\varphi_t+\varphi_t+nt)={\rm tr}_{\omega_t}\beta\geq 0
$$
and the maximum principle implies 
$$
(T'-t)\partial_t\varphi_t+\varphi_t+nt\geq T'\min_M\partial_t\varphi_{|t=0} =0 \,. 
$$ 
Since $\varphi$ is bounded, we get a lower bound for $\partial_t\varphi$.
 
\medskip 
\noindent $\bullet$ {\em ${\rm tr}_{\tilde \omega}\omega_t$ is uniformly bounded  from above in $[0,T_{\max})$.} 
In view of lemma \ref{estimate} we have 
$$
(\partial_t-\Delta_{B,t})\log({\rm tr}_{\tilde \omega}\omega_t)\leq C_3{\rm tr}_{\omega_t}\tilde \omega
$$
for a uniform constant $C_3$.  Let $A$ be a fixed constant such that $A\hat \omega_t-(C_3+1)\tilde \omega$ is a transversally K\"ahler form for every $t\in [0,T_{\max}]$. Then 
$$
{\rm tr}_{\omega_t}(C_3\tilde \omega-A\hat\omega_t)\leq -{\rm tr}_{\omega}\tilde \omega,\,\,\mbox{ in } M\times [0,T_{\max})\,.
$$
Hence
$$
(\partial_t-\Delta_{B,t})(\log({\rm tr}_{\tilde \omega}\omega_t)-A\varphi_t)\leq C_3{\rm tr}_{\omega_t}\tilde \omega-A\partial_t\varphi_t+A\Delta_{B,t}\varphi_t=
{\rm tr}_{\omega}(C_3\tilde \omega-A\hat \omega_t)-A\partial_t\varphi_t+An
$$
implies 
\b
\label{pre2}
(\partial_t-\Delta_{B,t})(\log({\rm tr}_{\tilde \omega}\omega_t)-A\varphi_t)\leq -{\rm tr}_{\omega_t}\tilde \omega -A\partial_t\varphi_t+An\,.
\e
Let $\tau\in (0,T_{\max})$ be fixed and let $(x_0,t_0)$ be a point where $\log ({\rm tr}_{\tilde \omega}\omega)-A\varphi$ achieves the maximum in $M\times[0,\tau]$.  If $t_0=0$, then 
$$
\max_{M\times [0,\tau]}(\log({\rm tr}_{\tilde \omega}\omega)-A\varphi)\leq \log n\,.
$$
If $t_0>0$, then \eqref{pre2} implies 
$$
{\rm tr}_{\omega_{t_0}}\tilde \omega\leq An -A\partial_t\varphi_{|t=t_0}= An-A\log\frac{\omega_{t_0}^n}{\Omega}\quad \mbox{ at } x_0\,,
$$
i.e. 
$$
{\rm tr}_{\omega_{t_0}}\tilde \omega +A\log\frac{\omega_{t_0}^n}{\Omega}\leq An\,\,\mbox{ at } x_0\,,
$$
and lemma \ref{LAlemma} implies that ${\rm tr}_{\tilde \omega}\omega$ is uniformly bounded  in $(x_0,t_0)$.  Hence, since $\|\varphi\|_{C^0}$ is uniformly bounded, we have 
$$
\max_{M\times [0,\tau]}(\log({\rm tr}_{\tilde \omega}\omega)-A\varphi)\leq (\log{\rm tr}_{\tilde \omega}\omega_{t_0})(x_0)+A\|\varphi_{t_0}\|_{C^0}\leq C
$$
where $C$ does not depend on $\tau$. Thus $\log({\rm tr}_{\tilde \omega}\omega)$ is uniformly bounded from above in $[0,T_{\max} )$ and the claim follows. 

\medskip 
  The three facts proved above together with lemma \ref{LAlemma} imply that $\frac{1}{C}\leq {\rm tr}_{\tilde \omega} \omega_t\leq C$ for a uniform constant $C$ and theorem \ref{thkahler} together with the compactness of $M$ implies  that the $C^{\infty}$ norm of $\omega$ is uniformly bounded in $M\times [0,T_{\max})$. Therefore as $t \to T_{\max}$ the solution $g_Q(t)$ converges to a transversally K\"ahler metric $g_Q(T_{\max})$ and the flow can be extended after $T_{\max}$ contradicting the maximality of the solution. 

\medskip 
In particular when $c_B^1(M)=0$, the maximal solution $g_Q$ is defined in $M\times [0,\infty)$.  Now we focus on this last case.
The fundamental form $\omega_t$ of $g_Q(t)$ can be written in this case as $\omega=\tilde \omega+i\partial_B\bar\partial_B \psi_t$, where $\psi_t$ solves 
$$
\partial_t\psi=\log\frac{(\tilde \omega+i\partial_B\bar\partial_B\psi_t)^n}{\tilde \omega^n}\,,\quad \tilde \omega+i\partial_B\bar\partial_B\psi_t>0\,,\quad \psi_{|t=0}=0
$$

\medskip
\noindent $\bullet$ {\em $\|\partial_t\psi_t\|_{C^0}$ is uniformly bounded in $[0,\infty)$.} The function $\psi$ solves $(\partial_t-\Delta_{B,t})\partial_t \psi_t=0$ and the maximum principle for basic maps implies this claim. 

\medskip
\noindent $\bullet$ {\em $\max\,\psi_t-\min\,\psi_t$ is uniformly bounded in $[0,\infty)$.} From the El Kacimi's paper \cite{EKA} it follows that the solutions to the transverse Monge-Amp\`ere equation 
$$
(\tilde \omega+i\partial_B\bar \partial_B f)^n=e^F\tilde \omega^n\,,\quad \tilde \omega+i\partial_B\bar\partial_Bf>0
$$
 satisfies the a priori estimate $\max f-\min f<C$ where $C$ depends only on $F$ and $\tilde \omega$. Now for every fixed $t$,  $\psi_t$ solves  
 $$
 (\tilde \omega+i\partial_B\bar\partial_B\psi_t)^n=\left(e^{\partial_t\psi_t}\right)\tilde \omega^n
 $$
 and the previous bound on $\partial_t\psi_t$ implies this claim. 
 
\medskip 
\noindent  {\em $\bullet$ ${\rm tr}_{\tilde \omega}\omega_t$ is uniformly bounded from above in $[0,\infty)$.} Lemma \ref{estimate} together with the compactness of $M$ implies that  $(\partial_t-\Delta_{B,t})\log{\rm tr}_{\tilde \omega}\omega_t\leq C\,{\rm tr}_{\omega_t} \tilde \omega$  for a uniform constant $C$.  It follows 
$$
\begin{aligned}
(\partial_t-\Delta_{B,t})(\log{\rm tr}_{\tilde \omega}\omega_t-(C+1)\psi_t)&\,\leq C\,{\rm tr}_{\omega_t} \tilde \omega-(C+1)\partial_t\psi_t-(C+1){\rm tr}_{\omega_t} \tilde \omega+Cn+n\\
&\,\leq -{\rm tr}_{\omega_t} \tilde \omega-(C+1)\partial_t\psi_t+Cn+n.
\end{aligned}
$$
Since $\partial_t\psi_t$ is uniformly bounded we have 
\b\label{pre}
(\partial_t-\Delta_{B,t})(\log{\rm tr}_{\tilde \omega}\omega_t-(C+1)\psi_t)\leq -{\rm tr}_{\omega_t} \tilde \omega+C_2
\e
for a uniform constant $C_2$. Now let us fix $\tau>0$ and  let $(x_0,t_0)$ be a point in $M\times [0, \tau]$ where $\log{\rm tr}_{\tilde \omega}\omega-(C+1)\psi$ achieves the maximum. If $t_0>0$, then
inequality \eqref{pre} implies ${\rm tr}_{\omega_{t_0}} \tilde \omega\leq C_2$ at $x_0$ and from 
$$
({\rm tr}_{\tilde \omega}\omega)\,\tilde \omega^n\wedge \chi\leq \frac{1}{(n-1)!}({\rm tr}_{\omega}\tilde \omega)^{n-1}\,\omega^{n}\wedge \chi=
\frac{1}{(n-1)!}({\rm tr}_{\omega}\tilde \omega)^{n-1}e^{\partial_t\psi}\,\tilde \omega^{n}\wedge \chi
$$
and the bound on  $\partial_t\psi$ it follows 
$$
{\rm tr}_{\tilde \omega}\omega_{t_0}\leq C_3\,\quad \mbox{ at }x_0\,,
$$ 
where $C_3$ does not depend on $\tau$. Moreover, since $\log{\rm tr}_{\tilde \omega}\omega-(C+1)\psi$ achieves the maximum at $(x_0,t_0)$, then we have 
$$
\log {\rm tr}_{\tilde \omega}\omega\leq C_3+(C+1)\psi-(C+1)\,\psi_{t_0}(x_0)\,,\,\, \mbox{ in }M\times [0,\tau]
$$
and so 
$$
\log {\rm tr}_{\tilde \omega}\omega\leq C_3+(C+1)\psi-(C+1)\,\min_{M\times [0,\tau]}\psi\,,\,\, \mbox{ in }M\times [0,\tau]\,.
$$
Let $V=\int_M \tilde \omega^n\wedge \chi$ and 
$$
\tilde \psi=\psi-\frac{1}{V}\int_M \psi\,\tilde \omega^n\wedge \chi\,.
$$
Then 
$$
\log {\rm tr}_{\tilde \omega}\omega\leq C_3+(C+1)\tilde \psi+\frac{C+1}{V}\int_M \psi\,\tilde \omega^n\wedge \chi
-(C+1)\,\inf_{M\times [0,\tau]}\tilde\psi-\frac{C+1}{V}\,\inf_{[0,\tau]}\int_M \psi\,\tilde \omega^n\wedge \chi
$$
in $M\times [0,\tau]$. Since $\tilde \psi$ is bounded in view of the previous point,
we get 
$$
\log {\rm tr}_{\tilde \omega}\omega\leq C_4+\frac{C+1}{V}\int_M \psi\,\tilde \omega^n\wedge \chi
-\frac{C+1}{V}\,\inf_{[0,\tau]}\int_M \psi\,\tilde \omega^n\wedge \chi\,,\quad \mbox{in }M\times [0,\tau]\,.
$$
Now 
$$
\frac{d}{dt}\int_M\psi\tilde \omega^n\wedge\chi=\int_M\log\left(\frac{\omega^n}{\tilde \omega^n}\right) 
\tilde \omega^n\wedge\chi \leq  V\log\left(\frac{1}{V}\int_M\omega^n\wedge \chi\right)=0
$$
shows that
$\int_M\psi\,\tilde \omega^n\wedge\chi$ is decreasing in $t$ and thus for every $(x,\tau)\in M\times [0,\infty)$ we have 
$$
\log {\rm tr}_{\tilde \omega}\omega_{\tau}(x)\leq C_4+\frac{C+1}{V}\int_M \psi_{\tau}\,\tilde \omega^n\wedge \chi
-\frac{C+1}{V}\,\int_M \psi_\tau\,\tilde \omega^n\wedge \chi= C_4\,.
$$
On the other hand if $t_0=0$, we have 
$$
\log{\rm tr}_{\tilde \omega}\omega_{t_0}(x_0)\leq \log n+(C+1)\psi_{0}(x)
$$
and we can prove the item by working in the same way. 

\medskip 
 Now, since the $C^0$-norm of $\partial_t\varphi$ is uniformly bounded, taking into account that $M$ is compact, the previous item and lemma  
\ref{LAlemma} imply $\frac{1}{C}\tilde \omega\leq \omega\leq C\omega$ in $[0,\infty)$ for a uniform constant $C$; therefore \ref{thkahler} implies that the $C^{\infty}$
norm of $\omega$ is uniformly bounded in $[0,\infty)$. It follows that $\psi_t$ converges to a basic smooth map $\psi_{\infty}$ on $M$ such that 
$\omega_{\infty}=\tilde \omega+i\partial_B\bar\partial_B\psi_{\infty}>0$. 

It remains to prove that $\omega_{\infty}$ has vanishing transverse Ricci tensor. Let $\tilde \rho_B$ be the basic Ricci form of $\tilde \omega$ and let $h\in C^{\infty}_B(M,\R)$ be such that  $\tilde \rho_B=i\partial_B\bar\partial_B h$. A direct computation yields that if  
$$
f(t)=\int_M\log \frac{\omega^n_t\wedge\chi}{\tilde \omega^n\wedge\chi}\,\omega^n_t\wedge\chi-\int_Mh(\omega^n_t-\tilde \omega^n)\wedge \chi\,,
$$ 
then 
\b\label{derivative}
\dot{f}(t)=-(\partial_B\dot \psi_t,\bar\partial_B\dot \psi_t)_{\omega_t}\,,\quad \ddot{f}(t)\leq 0\,,
\e
where $(\cdot,\cdot)_{\omega_t}$ is the scalar product \eqref{scalar} computed with respect to $\omega_t$. Equations \eqref{derivative} implies that $\dot f(t)\rightarrow 0$ as $t\rightarrow \infty$. Since $\partial_t\psi_t=\log\frac{(\tilde \omega+i\partial_B\bar\partial_B\psi_t)^n}{\tilde \omega^n}$, we obtain that $\partial_B \log\frac{\omega_{\infty}^n}{\tilde \omega^n}$ is constant which implies $\rho_B(\omega_{\infty})=0$. 
\end{proof}

\subsection{Proof of theorem \ref{c1<0}}
We may assume $\nu=1$ without loss of generality. About the short time existence it is enough to observe that if $g_Q$ solves \eqref{TKRF}, then the fundamental form of $\frac{1}{e^t}\,g_Q(e^t-1)$ solves \eqref{mTKRF}. Therefore \eqref{mTKRF} has a unique solution $\omega$ defined in $[0,\infty)$. By using the transverse $\partial_B\bar\partial_B$-lemma, we can write $\omega_t=\tilde \omega+i\partial_B\bar\partial_B \varphi_t$, where $\varphi_t$ solves 
$$
\partial_t\varphi_t=\log\frac{(\tilde \omega+i\partial_B\bar\partial_B\varphi_t)^n}{\tilde \omega^n}-\varphi_t\,,\quad \tilde \omega+i\partial_B\bar\partial_B\varphi_t>0\,,\quad \varphi_{|t=0}=0
$$

Now we have 

\medskip 
\noindent $\bullet$ {\em $\|\partial_t\varphi_t\|_{C^0}\leq Ce^{-t}$ for a uniform constant $C$.} 
Since $\partial^2_t\varphi_t=\Delta_{B,t}(\partial_t\varphi_t)-\partial_t\dot\varphi_t$, then $\partial_t(e^{t}\varphi_t)=\Delta_{B,t}(e^t\varphi_t)$ and  
the transverse maximum principle implies the item.

\medskip
Using lemma  \ref{phiK} together with the compactness of $M$ we have that there exists a basic smooth map $\varphi_{\infty}$
such that $\|\varphi_t-\varphi_{\infty}\|_{C^{0}}\leq C' e^{-t}$ and  $\frac{1}{C}\tilde \omega^n\leq \omega_t\leq C\tilde \omega^n$ for uniform constants
$C'$ and $C$. Moreover we have

\medskip 
\noindent {\em $\bullet$ ${\rm tr}_{\tilde \omega}\omega_t$ is uniformly upper bounded.} Lemma \ref{estimate} and the compactness of $M$ imply that 
$$
(\partial_t-\Delta_{B,t})(\log {\rm tr}_{\tilde \omega}\omega_t-(C+1)\varphi_t)\leq - {\rm tr}_{\omega_t}\tilde \omega-1-(C+1)\partial_t\varphi_t+(C+1)n
$$
for a uniform constant $C$. Let $\tau>0$ be fixed and let $(x_0,t_0)$ be a point in $M\times [0,\tau]$ where $\log {\rm tr}_{\tilde \omega}\omega-(C+1)\varphi$ achieves the maximum. If $t_0=0$, then 
$$
\log {\rm tr}_{\tilde \omega}\omega-(C+1)\varphi\leq \log n\mbox{ in }M\times [0,\tau]
$$
and therefore 
$$
\log {\rm tr}_{\tilde \omega}\omega\leq \log n+(C+1)\|\varphi\|_{C^0} \mbox{ in }M\times [0,\tau]\,.
$$
On the other hand, if $t_0>0$, then
$$
{\rm tr}_{\omega_{t_0}}\tilde \omega \leq -1-(C+1)\partial_t\varphi_{|t_0}+(C+1)n\,\,\ \mbox{ at }x_0\,,
$$
and therefore ${\rm tr}_{\omega_{t=}}\tilde \omega(x_0)$ is uniformly bounded in $(x_0,t_0)$. Since 
$$
({\rm tr}_{\tilde \omega}\omega_{t_0})\,\tilde \omega^n\wedge \chi\leq \frac{1}{(n-1)!}({\rm tr}_{\omega_{t_0}}\tilde \omega)^{n-1}\,\omega_{t_0}^{n}\wedge \chi,
$$
${\rm tr}_{\tilde \omega}\omega$ is uniformly bounded in $(x_0,t_0)$ and since $\varphi$ is uniformly bounded we get the item. 
 
\medskip
\noindent The two items above imply that the maximal solution $\omega_t=\tilde \omega+i\partial_B\bar\partial_B\varphi_t$ to \eqref{mTKRF} satisfies $\frac{1}{C}\tilde \omega\leq \omega_t\leq C\omega$ for a uniform constant and theorem \ref{thkahler} ensures that the $C^{\infty}$ norm of $\omega_t$ and of $\varphi_t$ are  uniformly bounded. 
This implies that $\omega_t$ converges to a transverse K\"ahler-Einstein structure, as required.

\subsection{The case of Sasaki manifolds.}
In the case of Sasaki metrics, theorem \ref{timedomain} and theorem \ref{c1<0} provide an alternative proof of the main results of \cite{SWZ} on Sasaki-Ricci flow. 

We recall that a Sasaki structure on a $(2n+1)$-dimensional manifold is given by a 1-dimensional foliation generated by a vector field $\xi$ together with the following triple of tensors: a bundle-like metric $g$, a 1-form $\eta$ such that $\ker \eta = \xi^\perp$ and an endomorphism $\Phi$ of $TM$ such that $\Phi^2=-{\rm Id}+\eta\otimes \xi$.
We denote by $\mathcal D$ the kernel of $\eta$ and by $g^T$ the restriction of $g$ to $\mathcal D$.  Clearly the pair $(\mathcal D,g^T)$ is identified with $(Q,g_Q)$.
If $g_Q(t)$ is a solution of the flow
$$\partial_tg_Q(t)=-{\rm Ric}_Q(t)-\nu g_Q(t) \,,\quad g_Q(0)=\tilde g_Q$$
with $\nu=0,-1$, we can reconstruct the Sasaki structure at any time by setting $g(t):=g^T(t)+\xi\otimes \xi$ and taking $\eta$ as the $g(t)$-dual of $\xi$.

\end{document}